\documentclass[a4paper,12pt]{amsart}

\author{Giles Gardam}
\title{Finite groups of arbitrary deficiency}

\usepackage[T1]{fontenc}
\usepackage[utf8]{inputenc}
\usepackage{amsfonts}
\usepackage{amsmath}
\usepackage{amssymb}
\usepackage{amsthm}
\usepackage{fullpage}
\usepackage{hyperref}
\usepackage{mathpazo}
\usepackage{mathtools}
\usepackage{parskip}
\usepackage{tikz}

\usepackage{comment}
\excludecomment{collectors}
\usepackage{setspace}
\theoremstyle{plain}
\newtheorem{thm}{Theorem}

\newtheorem{lem}{Lemma}

\theoremstyle{definition}
\newtheorem{defn}[lem]{Definition}
\newtheorem{eg}[lem]{Example}
\newtheorem{qn}[lem]{Question}
\newtheorem{rmk}[lem]{Remark}
\newtheorem*{thmA}{Theorem \ref{thm:main}}

\newcommand{\Gp}{\mathcal{G}_p}
\newcommand{\N}{\mathbb{N}}
\newcommand{\Q}{\mathbb{Q}}
\newcommand{\Z}{\mathbb{Z}}
\newcommand{\colonset}[2]{\left\{ #1 \, : \, #2 \right\}}
\newcommand{\df}{\operatorname{def}}
\newcommand{\gp}[2]{\langle \, #1 \, | \, #2 \, \rangle}
\newcommand{\inthalf}[1]{\lfloor \frac{#1}{2} \rfloor}
\newcommand{\mytimes}{{\mkern-2mu\times\mkern-2mu}}

\mathtoolsset{centercolon}

\thanks{This work was partially supported by the Clarendon Fund, Balliol College Marvin Bower Scholarship, and the James Fairfax Oxford Australia Scholarship}
\address{Mathematical Institute \\ Andrew Wiles Building \\ University of Oxford \\ Oxford, OX2~6GG \\ United Kingdom}
\curraddr{Department of Mathematics \\ Technion \\ Haifa \\ Israel}
\email{gilesgar@technion.ac.il}

\begin{document}

\begin{abstract}
    The deficiency of a group is the maximum over all presentations for that group of the number of generators minus the number of relators.
    Every finite group has non-positive deficiency.
    We show that every non-positive integer is the deficiency of a finite group -- in fact, of a finite $p$-group for every prime $p$.
    This completes Kotschick's classification \cite{kotschick} of the integers which are deficiencies of fundamental groups of compact Kähler manifolds.
\end{abstract}

\maketitle

\section{Introduction}
The \emph{deficiency} of a finitely presented group is the maximum over all presentations for that group of the number of generators minus the number of relators (some authors use the opposite sign convention).
Every finite group has non-positive deficiency, since a group of deficiency at least $1$ has infinite abelianization.
For finite groups, most recent study of deficiency has focussed on finding deficiency zero presentations.
The celebrated work of Golod and Shafarevich implies that a finite $p$-group of rank $d$ has deficiency less than $-\frac{d^2}{4}+d$; this is one of many asymptotic results on deficiency of finite groups.
On the other hand, the range of techniques for determining deficiencies of groups precisely is very limited.
Specifically, the literature does not appear to contain a proof that all negative integers arise as deficiencies of finite groups.
(The \emph{infinite} case is easy: every integer is the deficiency of some $F_r * \Z^s$.)
Well-known examples of finite groups achieving \emph{arbitrarily large} negative deficiency, such as abelian groups and Swan's examples with trivial Schur multiplier \cite{swan}, have quadratically growing deficiency: the set of their negative deficiencies has density zero in $\N$.
Another example of our lack of understanding of the fine structure of deficiency is an open problem in the Kourovka Notebook \cite[8.12(a)]{kourovka}, due to D.~L.~Johnson and E.~F.~Robertson:  Does there exist a finite $p$-group of rank $3$ and deficiency zero for any $p \geq 5$?
For rank $d \geq 4$ no such finite $p$-group exists, for any prime $p$, by Golod--Shafarevich.

In this article we prove the following theorem, which shows that indeed all negative integers arise as deficiencies of finite groups.
The finite groups $A_p$, $B_p$ and $C_p$ -- which are parameterized by a prime $p$ -- are introduced in Definition~\ref{defn:3ingredients}.
\begin{thm}
    \label{thm:main}
    Let $p$ be a prime, and $n \in \N$.
    Then there are natural numbers $r$, $s$ and $t$ such that the finite $p$-group $A_p^r \times B_p^s \times C_p^t$ has deficiency $-n$.
\end{thm}

A \emph{Kähler group} is the fundamental group of a compact Kähler manifold; such a group is always finitely presented.
The class of Kähler groups includes all finite groups \cite{serre}, as well as surface groups and more generally the fundamental groups of complex projective varieties.
Kotschick proved in \cite{kotschick} that no Kähler group has even positive deficiency, and noted that this is the only constraint on positive deficiency for Kähler groups, as all odd positive integers arise as deficiencies of surface groups~$\Sigma_g$.
He then gave examples of Kähler groups of all negative deficiencies except for -5 and -7, with the suggestion that these should be achievable with finite groups (see Section~6 in~\cite{kotschick}).

Theorem~\ref{thm:main} completes, with proof, the classification of deficiencies of Kähler groups, as suggested by Kotschick.

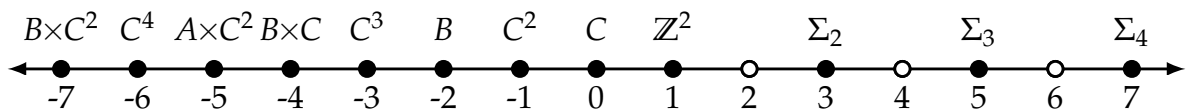
\begin{figure}[!h]

\begingroup \resizebox {\columnwidth} {!} {

\begin{tikzpicture}[thick,scale=0.75, every node/.style={scale=0.75}]
\draw[latex-latex] (0.3,0) -- (15.7,0);
\foreach \x  in {1,...,15}
  \draw[xshift=\x cm] (0pt,2pt) -- (0pt,-2pt) node[below,fill=white,text height=1.5ex, text depth=.25ex, text width=3em, text centered] {\the\numexpr\x -8\relax};
\node[text height=1.5ex, text depth=.25ex, text width=3em, text centered] at (1, 0.5) {$B \mytimes C^2$};
\node[text height=1.5ex, text depth=.25ex, text width=3em, text centered] at (2, 0.5) {$C^4$};
\node[text height=1.5ex, text depth=.25ex, text width=3em, text centered] at (3, 0.5) {$A \mytimes C^2$};
\node[text height=1.5ex, text depth=.25ex, text width=3em, text centered] at (4, 0.5) {$B \mytimes C$};
\node[text height=1.5ex, text depth=.25ex, text width=3em, text centered] at (5, 0.5) {$C^3$};
\node[text height=1.5ex, text depth=.25ex, text width=3em, text centered] at (6, 0.5) {$B$};
\node[text height=1.5ex, text depth=.25ex, text width=3em, text centered] at (7, 0.5) {$C^2$};
\node[text height=1.5ex, text depth=.25ex, text width=3em, text centered] at (8, 0.5) {$C$};
\node[text height=1.5ex, text depth=.25ex, text width=3em, text centered] at (9, 0.5) {$\Z^2$};
\node[text height=1.5ex, text depth=.25ex, text width=3em, text centered] at (11, 0.5) {$\Sigma_2$};
\node[text height=1.5ex, text depth=.25ex, text width=3em, text centered] at (13, 0.5) {$\Sigma_3$};
\node[text height=1.5ex, text depth=.25ex, text width=3em, text centered] at (15, 0.5) {$\Sigma_4$};

\node[fill=black,draw=black,circle,inner sep=2pt] at (1,0) {};
\node[fill=black,draw=black,circle,inner sep=2pt] at (2,0) {};
\node[fill=black,draw=black,circle,inner sep=2pt] at (3,0) {};
\node[fill=black,draw=black,circle,inner sep=2pt] at (4,0) {};
\node[fill=black,draw=black,circle,inner sep=2pt] at (5,0) {};
\node[fill=black,draw=black,circle,inner sep=2pt] at (6,0) {};
\node[fill=black,draw=black,circle,inner sep=2pt] at (7,0) {};
\node[fill=black,draw=black,circle,inner sep=2pt] at (8,0) {};
\node[fill=black,draw=black,circle,inner sep=2pt] at (9,0) {};
\node[fill=black,draw=black,circle,inner sep=2pt] at (10,0) {};
\node[fill=black,draw=black,circle,inner sep=2pt] at (11,0) {};
\node[fill=black,draw=black,circle,inner sep=2pt] at (12,0) {};
\node[fill=black,draw=black,circle,inner sep=2pt] at (13,0) {};
\node[fill=black,draw=black,circle,inner sep=2pt] at (14,0) {};
\node[fill=black,draw=black,circle,inner sep=2pt] at (15,0) {};

\node[fill=white,draw=black,circle,inner sep=2pt] at (10,0) {};
\node[fill=white,draw=black,circle,inner sep=2pt] at (12,0) {};
\node[fill=white,draw=black,circle,inner sep=2pt] at (14,0) {};
\end{tikzpicture}
}
\endgroup
\caption{Deficiencies of Kähler groups}
\end{figure}

A search of the literature on deficiencies of finite groups suggests that one can extract examples as needed by Kotschick from the work of Sag and Wamsley, who claimed to have computed the deficiency of every group of order $2^n$ for $n \leq 6$ \cite{sagwamsley}.
However, they did not publish proofs, and the article does contain a number of errors beyond the obvious misprints: some presentations are not \emph{efficient} as claimed, and others do not define the groups they should.
(To give one concrete example, the 252\textsuperscript{nd} presentation of a group of order 64 is in fact a presentation of $\Z/4 \rtimes \Z/4$, for either commutator convention.)
\begin{collectors}
In future work, we will give a thorough analysis of their article and what is known about the deficiencies of small $2$-groups.
\end{collectors}

The outline of our broad strategy to construct finite groups of arbitrary deficiency -- and of the structure of this article -- is as follows.
We introduce in Section~\ref{section:control} the class $\Gp$ of efficient $p$-groups, in which we fully understand the deficiency of direct products (a quadratic polynomial).
After finding enough basic examples in Section~\ref{section:buildingblocks}, we take repeated direct products in suitable combinations to obtain all negative integers as deficiencies; this analysis is the topic of Section~\ref{section:construction}.
In analogy with results on the representability of positive integers by quadratic forms, such as the Conway--Schneeberger Fifteen Theorem \cite{bhargava}, there is good reason to expect such a strategy to succeed if we find enough basic building blocks in the class $\Gp$.
The fact that the deficiency of our direct product is an inhomogeneous quadratic polynomial, rather than a quadratic form, makes the analysis easier, and we in fact only need the three basic examples $A_p$, $B_p$ and $C_p$.

\section{Controlling deficiency}
\label{section:control}

For a group $G$, let $d(G)$ denote the minimal size of a generating set for $G$, which we call the \emph{rank} of $G$.
The homology groups $H_*(G)$ are implicitly taken with trivial $\Z$ coefficients.
In particular, the abelianization~$G / G'$ is isomorphic to $H_1(G)$.
The deficiency of a group $G$ is bounded above by
\begin{equation}
    \label{align:inequality}
    \tag{$\star$}
    \df(G) \leq \operatorname{rk}(H_1(G)) - d(H_2(G))
\end{equation}
where $\operatorname{rk}$ denotes the torsion-free rank of an abelian group: $\operatorname{rk}(G) = \operatorname{rk}_\Q (G \otimes_\Z \Q)$.
For a proof of this well-known inequality, the reader is referred to, for example, \cite[Lemma 2]{bridsontweedale} (NB: that article uses the opposite sign convention for deficiency).
If a group achieves equality in \eqref{align:inequality}, then it is called \emph{efficient}.
The torsion-free rank of every finite group is zero, so the upper bound on deficiency of a finite group is simply minus the rank of the Schur multiplier $H_2(G)$.

A presentation realizing the deficiency of a group $G$ is called \emph{minimal} if it moreover has the minimal possible number of generators, namely the rank $d(G)$.
Note that we are asking more of a `minimal' presentation than other authors; for example, \cite[\S 4]{gruenberg_relation_modules} only requires the number of generators of the group to be $d(G)$ with no requirements on the number of relators.
A group can admit a minimal presentation without being efficient (for example, Lustig's non-efficient torsion-free example $\gp{x,y,z}{x^2=y^3, [x,z], [y,z]}$ \cite{lustig} actually admits the minimal presentation $\gp{a,b}{[a, b^3], [a^2, b]}$, with an isomorphism given by $a \mapsto xz$, $b \mapsto yz$).
Indeed, it is an open problem whether every group admits a minimal presentation.
Rapaport proved \cite{rapaport} that this is the case for one-relator groups and nilpotent groups.

One class of finite groups where deficiency is reasonably understood is the class $\Gp$, for $p$ a prime, as defined in \cite{johnson}.

\begin{defn}
    The class $\Gp$ denotes the finite $p$-groups $G$ such that $G$ is efficient and admits a minimal presentation.
\end{defn}

In particular, the number of relators of such a presentation is simply $d(H_1(G)) + d(H_2(G))$, since every finite $p$-group $G$ (indeed, every nilpotent group) satisfies $d(G) = d(H_1(G))$.
(By Rapaport's theorem, one could remove the requirement of admitting a minimal presentation from the definition of $\Gp$.)

In fact, there is no known example of a non-efficient finite $p$-group.
\begin{qn}[{\cite[Question 18]{mann:questions}}]
    Is every finite $p$-group an element of $\mathcal{G}_p$?
\end{qn}

The class $\Gp$ has been shown to be closed under various operations.
For our purposes here, we only need closure under direct products as proved in \cite{johnson}.
Since it is short and instructive, we include here a proof of this fact.

\begin{lem}
    Let $G, H \in \Gp$.
    Then $G \times H \in \Gp$.
    Moreover, if minimal presentations are $G = \gp{X}{R}$, $H = \gp{Y}{S}$, then a minimal presentation for $G \times H$ is \[
        \gp{X \sqcup Y}{R \sqcup S \sqcup \colonset{[x,y]}{x \in X, y \in Y}}.
    \]
\end{lem}

\begin{proof}
    The above is a finite presentation of the finite $p$-group $G \times H$, and it has the required number of generators as $d(G \times H) = d(G) + d(H)$, since $G$ and $H$ are finite $p$-groups.
    It thus remains to prove that this is an efficient presentation, that is, that $G \times H$ has deficiency $-d(H_2(G \times H))$.

    Recall that all the homology groups of a finite $p$-group are finite abelian $p$-groups.
    The special case of the Künneth formula proved by Schur states that \[
        H_2(G \times H) \cong H_2(G) \oplus H_2(H) \oplus (H_1(G) \otimes_\Z H_1(H)).
    \]
    As all four terms on the right-hand side are finite abelian $p$-groups, we see that
    \begin{align*}
        d(H_2(G \times H)) &= d(H_2(G)) + d(H_2(H)) + d(H_1(G)) \cdot d(H_1(H)) \\
                           &= |R| - |X| + |S| - |Y| + |X| \cdot |Y|
    \end{align*}
    and thus the presentation is efficient.
\end{proof}

\begin{collectors}
\begin{rmk}
    As pointed out to us by Derek Holt, it is not difficult to attain finite groups of arbitrary (finite abelian) Schur multiplier.
    For every $n \geq 2$, if we choose $q$ to be a prime power congruent to $1$ modulo $n$, then $\operatorname{PSL}(n, q)$ is simple and has Schur multiplier cyclic of order $n$.
    (In the case $n = 2$ we require additionally $q \neq 3,9$ and for $n = 3$ we need $q \neq 4$.)
    These groups are all perfect, so the Schur multiplier of a direct product of any of them will simply be the direct product of the Schur multipliers.
    Being perfect (and moreover simple) these groups are as far from the classes $\Gp$ as one could imagine for finite efficient groups, and there is no obvious way to construct deficiency zero presentations of their direct products.
\end{rmk}
\end{collectors}

\section{Building blocks}
\label{section:buildingblocks}

Fix a prime $p$. To construct our desired groups of arbitrary negative deficiency, we only need the following three groups from $\Gp$. Beyond being members of $\Gp$, the only relevant property of these basic examples is that the number of generators and relators in their minimal presentations are the pairs $(2,2), (2,4)$ and $(1,1)$.

\begin{defn}
    \label{defn:3ingredients}
    Define groups by the presentations
    \begin{alignat*}{2}
        A_p &:= \gp{a,b}{a^p = b^p, \,a^b = a^{p+1}} \\
        B_p &:= \gp{a,b}{a^p, \,b^p, \,[[a,b], a], \,[[a,b], b]} \\
        C_p &:= \gp{a}{a^p}
    \end{alignat*}
    except when $p = 2$, where we define $B_2 := \gp{a,b}{a^4, b^4, (ab)^2, (a^{-1}b)^2}$.
\end{defn}

\begin{lem}
    $A_p$, $B_p$ and $C_p$ are all elements of $\Gp$ and the above presentations are minimal.
\end{lem}

\begin{proof}
First note that for each of the three presentations, the number of generators equals the rank of the abelianization, so minimality will follow once we establish that the presentations achieve the deficiency of their respective groups.
Both $A_p$ and $C_p \cong \Z/p$ are given by deficiency zero presentations, so to show they are elements of $\Gp$ it remains only to show that $A_p$ is a finite $p$-group.
(In fact, $A_2 \cong Q_8$ and $A_p \cong \Z/p^2 \rtimes \Z/p$ for odd $p$, but this is not needed for the proof.)

In $A_p$, the relation $a^b = a^{p+1}$ can be written as $[a, b] = a^p$, so since $a^p = b^p$ the commutator $[a, b]$ is central.
Thus $[a, b]^p = [a^p, b] = [b^p, b] = 1$, so $a^{p^2} = (a^p)^p = [a, b]^p = 1$, and likewise $b^{p^2} = 1$, so $A_p$ is nilpotent and generated by $p$-torsion, hence a finite $p$-group.

Table 8.1 in \cite{karpilovsky} lists $B_2$, of order 16, as $G_{15}$, with Schur multiplier $(\Z/2)^2$, as proved in \cite{tahara}.
For odd $p$, $B_p$ is the mod-$p$ Heisenberg group (of order $p^3$ and exponent $p$), with Schur multiplier $(\Z/p)^2$ \cite[4.16]{beyltappe}.
Thus $B_p \in \Gp$.
\end{proof}

\section{The construction}
\label{section:construction}

We can now prove the main theorem, which we recall for the convenience of the reader.

\begin{thmA}
    Let $p$ be a prime, and $n \in \N$.
    Then there are natural numbers $r$, $s$ and $t$ such that the finite $p$-group $A_p^r \times B_p^s \times C_p^t$ has deficiency $-n$.
\end{thmA}

\begin{proof}
    We first compute the deficiency of $A_p^r \times B_p^s \times C_p^t$.
    Recall that the number of generators and relators in minimal presentations for these groups are $(2, 2)$, $(2, 4)$ and $(1, 1)$ respectively.
    The standard presentation for their direct product will have $2r + 2s + t$ generators.
    There are $2r + 4s + t$ relators coming from the relators of each direct factor, and $\binom{2r + 2s + t}{2} - r - s$ introduced commutativity relators: one for each unordered pair of generators, except those pairs that belong to the same direct factor.
    Thus we need to find $r, s, t \in \N$ such that \[
        \binom{2r + 2s + t}{2} + s - r = n.
    \]

    Let $m$ be the smallest positive integer such that $\binom{m}{2} + \inthalf{m} \geq n$, and let $d := n - \binom{m}{2} \leq \inthalf{m}$.
    By choice of $m$, we have $\binom{m-1}{2} + \inthalf{m-1} \leq n$ (with equality only if $m=1$).
    As $\inthalf{m-1} + \inthalf{m} = m-1$ and $\binom{m-1}{2} + m-1 = \binom{m}{2}$, we have $\binom{m-1}{2} + \inthalf{m-1} = \binom{m}{2} - \inthalf{m}$, so $d \geq -\inthalf{m}$.

    If $d \geq 0$, let $s := d$ and $r := 0$, and if $d < 0$ let $r := -d, s := 0$, so that in either case $s-r = d$ and $r+s \leq \inthalf{m}$.
    Now we can let $t := m - 2r - 2s$, and thus $\binom{2r + 2s + t}{2} + s - r = \binom{m}{2} + d = n$ as required.
\end{proof}

\begin{eg}
    The group $A_p \times C_p^2$ has deficiency~$-5$, and $B_p \times C_p^2$ has deficiency~$-7$.
\end{eg}

\begin{rmk}
    There are infinitely many alternatives for the groups $A_p$, $B_p$ and $C_p$ with the same numbers of generators and relators: $(2, 2), (2, 4), (1, 1)$.
    In fact, there are choices with \emph{different} generator-relator pairs for which Theorem~\ref{thm:main} holds (the combinatorial argument in the proof is then different).
    At least for $p=2$ and $p=3$ there are $p$-groups with a minimal 3-generator 3-relator presentation \cite[\S 4]{johnsonrobertson}.
    We can replace $B_p$ by such a group, and find $p$-groups of arbitrary deficiency as direct products of deficiency zero groups with generator-relator pairs $(1, 1), (2, 2), (3, 3)$.
    For an alternative with pairs $(1, 1), (2, 4), (2, 5)$, we could replace $A_p$ with \(
        \gp{a,b}{a^{p^2}, \,b^{p^2}, \,[[a,b], a], \,[[a,b], b], \,[a, b]^p}
    \) which is order $p^5$ and has Schur multiplier $(\Z/p)^3$ (proving this, however, requires work).
    For a detailed analysis of variations on the construction, see \cite{thesis}.
\end{rmk}

\subsection*{Acknowledgements} I would like to thank my supervisor Martin Bridson for guidance, encouragement, and editing. I would also like to thank Inna Capdeboscq and Ian Leary for helpful conversations, Derek Holt for helpful correspondence, and Claudio Llosa Isenrich for suggesting the problem and helpful conversations thereafter.

\vspace{-3ex}
\bibliography{deficiencies}{}
\bibliographystyle{alpha}

\end{document}